 \documentclass[reqno,11pt]{amsart}
 \usepackage{amsmath, epsfig, amsthm, a4, latexsym, amssymb}
 \usepackage{bbm}

\def\@rmrk#1#2{\refstepcounter
    {#1}\@ifnextchar[{\@yrmrk{#1}{#2}}{\@xrmrk{#1}{#2}}}

\makeatletter\@addtoreset{equation}{section}\makeatother

 \newfont{\bfit}{cmbxti10 scaled 1200}

 \newcommand{\eps}{\varepsilon}
 \newcommand{\im}{\mathsf i}

 \newcommand{\ball}{\mathcal{B}}
 \newcommand{\annulus}{\mathcal{A}}
 \renewcommand{\circle}{\mathcal{C}}

 \newcommand{\skris}{{\mathcal S}}

 \newcommand{\sfrac}[2]{\mbox{$\frac{#1}{#2}$}}
 \newcommand{\ssup}[1] {{\scriptscriptstyle{({#1}})}}

\newcommand{\N}{\mathbbm{N}}       % Natuerliche Zahlen
\newcommand{\R}{\mathbbm{R}}       % Reelle Zahlen
\newcommand{\C}{\mathbbm{C}}       % Komplexe Zahlen
\newcommand{\E}{\mathbbm{E}}       % Definitionsmenge
\renewcommand{\H}{\mathbbm{H}}     % Halbebene
\newcommand{\one}{\mathbbm{1}}     % Einheits Eins
\newcommand{\prob}{\mathbbm{P}}    % Wahrscheinlichkeit
\renewcommand{\subsection}{\secdef \subsct\sbsect}
\newcommand{\subsct}[2][default]{\refstepcounter{subsection}
\vspace{0.15cm}
{\flushleft\bf \arabic{section}.\arabic{subsection}~\bf #1  }
\nopagebreak\nopagebreak}
\newcommand{\sbsect}[1]{\vspace{0.1cm}\noindent
{\bf #1}\vspace{0.1cm}}

{\nopagebreak {\hfill\rule{2mm}{2mm}}\\ }

\newtheorem{theorem}{Theorem}[section]
\newtheorem{lemma}[theorem]{Lemma}
\newtheorem{definition}[theorem]{Definition}

\newtheoremstyle{thm}{1.5ex}{1.5ex}{\itshape\rmfamily}{}
{\bfseries\rmfamily}{}{2ex}{}

\newtheoremstyle{rem}{1.3ex}{1.3ex}{\rmfamily}{}
{\itshape\rmfamily}{}{1.5ex}{}
\theoremstyle{rem}
\newtheorem{remark}{{\slshape\sffamily Remark}}[]

\refstepcounter{subsubsection}

\mathchardef\ordinarycolon\mathcode`\:
\mathcode`\:=\string"8000
\begingroup \catcode`\:=\active
  \gdef:{\mathrel{\mathop\ordinarycolon}}
\endgroup
\def\thebibliography#1{\section*{Bibliography}
  \list%
  {\arabic{enumi}.}%                          *** style of reference number ***
    {\settowidth\labelwidth{[#1]}\leftmargin\labelwidth
    \advance\leftmargin\labelsep
    \parsep0pt\itemsep0pt
    \usecounter{enumi}}
    \def\newblock{\hskip .11em plus .33em minus .07em}
    \sloppy                   % \clubpenalty4000\widowpenalty4000
    \sfcode`\.=1000\relax}

\begin{document}
%%%%%%%%%%%%%%%%%%%%%%%%%%%%%%%%%%%%%%%%%%%%%%%
\title[Double-points on the Brownian frontier]
{The Hausdorff dimension of the double points on the Brownian frontier}

{\LARGE \bf The Hausdorff dimension of the double points\\ on the Brownian frontier}

\thispagestyle{empty}
\vspace{0.2cm}

\textsc{Richard Kiefer}\\
Fachbereich Mathematik,
Universit\"at Kaiserslautern, %Postfach 3049
67653 Kaiserslautern, Germany\\
E--mail: \texttt{kiefer@mathematik.uni-kl.de}

\textsc{Peter M\"orters\footnote{Supported by an Advanced Research Fellowship of the EPSRC.}}\\
Department of Mathematical Sciences,
University of Bath, Bath BA2 7AY, England\\
E--mail: \texttt{maspm@bath.ac.uk}

\vspace{0.5cm}

% \centreline{\small(Draft of \version)}
% \vspace{0.5cm}

\begin{quote}{\small {\bf Abstract: }
The frontier of a planar Brownian motion is the boundary of the unbounded component of the 
complement of its range. In this paper we find the Hausdorff dimension of the set of double 
points on the frontier. }
\end{quote}

\begin{quote}{\small {\bf R\'esum\'e: }
Nous d\'eterminons la dimension de Hausdorff de l'ensemble des points doubles situ\'es
sur la fronti\`ere d'un mouvement brownien plan.}
\end{quote}
\vspace{0.5cm}

\begin{tabular}{lp{13cm}}
\multicolumn{2}{l}
{{\it MSC 2000}: Primary 60J65; Secondary 60G17. }\\
{\it Keywords:} & Brownian motion, self-intersections, double points, frontier, 
outer boundary, disconnection exponent, Mandelbrot conjecture, Hausdorff dimension.\\
\end{tabular}

% \tableofcontents
\vspace{0.5cm}

\section{Introduction and statement of the results}

Let $(W_s \colon 0\le s \le \tau)$ be a standard planar Brownian motion running 
up to the first hitting time~$\tau$ of the circle of unit radius around the origin,
and consider the complement of its path, i.e.
$$\big\{x\in \R^2 \colon x\not= W_s \mbox{ for any $0\le s \le \tau$} \big\}.$$
This set is open and can be decomposed into connected components, exactly one
of which is unbounded. We denote this component by $U$ and define its
boundary $\partial U$ as the \emph{frontier} of the Brownian path. Note that in 
this natural setup the frontier is a random closed curve enclosing the origin, 
which is contained in the unit disc and touches the unit circle in exactly one point. 
The frontier can be seen as the set of points on the Brownian path which are accessible from 
infinity and is therefore also called the \emph{outer boundary} of Brownian motion. 
\smallskip

Mandelbrot conjectured, based on a simulation and the analogy of the outer boundary
and the self-avoiding walk, that the Brownian frontier has Hausdorff dimension $4/3$, 
see~\cite{Ma82}. Rigorous confirmation of this conjecture, however, turned out to be a hard 
problem, which took a long time. In the late nineties Bishop, Jones, Pemantle and Peres~\cite{BP97} 
showed that the frontier has Hausdorff dimension strictly larger than one, and about the same 
time Lawler~\cite{La96} identified the Hausdorff dimension in terms of a (then) unknown constant, 
the disconnection exponent~$\xi(2)$. A few years later, Lawler, Schramm and Werner, as one of the 
first applications of their SLE technique, found the explicit value of this constant and thus 
confirmed Mandelbrot's conjecture.%
\smallskip

As a planar Brownian motion has points of any finite (and indeed infinite) multiplicity,
and these points form a dense set of full dimension on the range, it is natural 
to ask whether there are multiple points also on the frontier. % and if so, how many.
\smallskip
\pagebreak[3]

To begin with, it is easy to observe that the Brownian frontier must contain \emph{double
points} of the Brownian motion. The argument, which is due to L\'evy~\cite{Le65},
goes roughly like this: If there were no double points on the frontier, it would
by construction contain a stretch of the original Brownian path. This would however imply that 
it had double points, which is a contradiction.
%\smallskip 
%
Knowing that there are double points on the frontier, it is natural to ask, whether the 
frontier contains \emph{triple points}. This problem was solved by Burdzy and Werner~\cite{BW96},
who showed that, almost surely, there are \emph{no} triple points on the frontier of a planar Brownian motion.
\smallskip

A second natural question that comes up is \emph{how many} double points one can find on the
Brownian frontier. Maybe surprisingly, it turns out that while the set
$$D = \big\{ x\in \R^2 \colon x=W_s=W_t \mbox{ for distinct } s,t\in[0,\tau] \big\}$$ 
of double points has full Hausdorff dimension on the entire path, it does not have full 
dimension on the frontier. The following curious result is the main result of this paper.
\smallskip

\begin{theorem}\label{double}
Almost surely, the set of double points on the Brownian frontier satisfies
$$\dim \big( D \cap \partial U \big) = \frac{\sqrt{97}+1}{24}\, .$$
\end{theorem}
\smallskip

\begin{remark} By a variation of the proof one can see that the same formula holds when the 
Brownian motion is stopped at a fixed time, rather than the exit time~$\tau$ from the unit disc.
\end{remark}
\medskip

We would like to point out that our proof of Theorem~\ref{double} uses a technique 
different from that of~\cite{La96}. The latter paper works in the time domain and 
uses Kaufman's dimension doubling lemma, see e.g. \cite[Theorem~9.23]{MP08}, to move 
to the spatial domain. We found this approach not suitable to deal with the lower
bound in the case of double points, in particular as it would require rather delicate 
estimates for Brownian bridges. Instead, as in other problems related to double points, 
see e.g.~\cite{KM05}, it is preferable to work directly in the plane. In the next chapter 
we give an accessible sketch 
of the proof, and also state sharp estimates for disconnection probabilities, which are at 
the heart of our argument, and may be of independent interest, see Theorem~\ref{exp}. 
Full technical details of the proof are given in the final chapter.

\section{Proof of Theorem~\ref{double}: Framework and ideas.}

In the proof of this result, we consider Brownian motion up to the first exit time~$\tau$ from a 
disc around the origin of fixed radius~$R$, say larger than two, and adapt the definition of the set~$D$
of double points and the frontier~$U$ accordingly. We fix a compact square~$S_0$ of unit sidelength 
contained in this disc, which does not contain the origin, and $\delta>0$ smaller than 
half the distance 
of $S_0$ to the origin. Let $\mathcal S_n$ be the collection  containing those of the 
$2^{2n}$ nonoverlapping compact subsquares~$S\subset S_0$ of sidelength~$2^{-n}$ that satisfy
\begin{itemize}
\item the Brownian motion $(W_s \colon s \ge 0)$ hits the square $S$, 
moves a distance of order~$\delta$, and then hits $S$ again before 
the killing time~$\tau$;
\item the union of the paths outside the square~$S$ does not 
disconnect its boundary $\partial S$ from infinity. 
\end{itemize}
The exact definition is such that, for some integer~$N(\delta)\in\N$, when $N(\delta)\le m \le n$, and $S\in\mathcal S_n$ then all dyadic squares $T\supset S$ with sidelength $2^{-m}$ are in $\mathcal S_m$. We further have
$$S_0 \cap D \cap \partial U = \bigcup_{\delta>0} \, \bigcap_{n=N(\delta)}^\infty \, \bigcup_{S \in {\mathcal S}_n} 
\, S,$$
and the Hausdorff dimension can be determined with positive probability by verifying
first and second moment criteria: Let $\xi>0$ and assume that there exist constants $c_1,c_2,c_3>0$ 
such that
\begin{itemize}
\item[(i)] for any dyadic subsquare $S\subset S_0$ of sidelength $2^{-n}$, we have 
$$c_1 2^{-\xi n} \leq \prob\big(S \in {\mathcal S}_n\big)  \leq c_2 2^{-\xi n};$$
\item[(ii)]for any pair of dyadic subsquares $S,T\subset S_0$ of sidelength $2^{-n}$
with distance of order $2^{-m}$, $1\le m \le n$, we have
$$\prob\big(S,T \in {\mathcal S}_n\big) \leq c_3 \, 2^{-2\xi n} 2^{\xi m}.$$
\end{itemize}
These conditions imply that $\dim(S_0 \cap D \cap \partial U)\le 2-\xi$ almost surely
and $ \dim(S_0 \cap D \cap \partial U)\ge 2-\xi$ with positive probability,
see  \cite[Theorem 10.43]{MP08}. This is a standard technique (sometimes called \emph{second moment method})
in fractal geometry and not too hard to verify, for example using the mass distribution principle.
\smallskip%

To get hold of the constant~$\xi$ we first recall the definition of the 
\emph{disconnection exponents} of planar Brownian motion. Let
$\ball(z,r)$ the open disc of radius $r$ with centre~$z$ and suppose 
$(W^{\ssup i}_s \colon s \ge 0)$, for $i\in\{1,\ldots,k\}$, are independent Brownian motions
started on the unit circle $\partial\ball(0,1)$, and stopped upon leaving the 
concentric disc $\ball(0,e^n)$ of radius~$e^n$. 
We denote by ${\mathfrak B}_n$ the union of their paths, and by $V_n$ the event that 
the set ${\mathfrak B}_n$ does not disconnect the origin from infinity, i.e.\ 
the origin is in the unbounded connected component 
of the complement of~${\mathfrak B}_n$. The disconnection exponent $\xi(k)$
is then defined by the requirement that there exist positive constants $c_1$ and $c_2$ such that,
for any $n\in\N$,
\begin{equation}\label{exp_def}
c_1\, \exp\{-n\,\xi(k)\} \le \prob(V_n)\le c_2\, \exp\{-n\,\xi(k)\}\, . \\[1mm]
\end{equation}
Lawler~\cite{La96} showed that the disconnection exponents are well-defined by
this requirement, and Lawler, Schramm and Werner~\cite{LSW01} found the explicit values 
$$\xi(k)= \frac{(\sqrt{24k+1}-1)^2-4}{48}.$$ 

The intuition behind our proof is that locally the paths of a Brownian motion seen from a 
typical double point look like \emph{four} Brownian motions started at this point. Roughly
speaking, each of the two segments of the path crossing in the double point, is split into 
a part prior to hitting the double point, and a part after hitting the double point, amounting
to four paths altogether. Hence the probability that a disc or square of diameter~$2^{-n}$ 
containing this double point is not disconnected from infinity by these paths should be of 
order $2^{-n \xi(4)}$. In reality, things are a bit more delicate and this observation is 
only correct up to a factor, which is polynomial in~$n$. Indeed, when we place a small disc 
around a potential double point, and split a path, which is conditioned to hit this disc,
at the first hitting time, the time-reversal of the path up to this instant spends somewhat 
less time in the critical area near the disc and therefore non-disconnection probabilities 
are slightly larger than for Brownian motion starting on the circle. Here is the rigorous statement 
behind our argument.

\begin{theorem}\label{exp}
Suppose  $(W^{\ssup i}_s \colon s \ge 0)$ for $i\in\{1,\ldots, k\}$ are independent Brownian motions
started uniformly on the circle $\partial\ball(0,\frac12\,e^n)$, and stopped upon leaving the disc $\ball(0,e^{n})$, 
$n\ge 2$, i.e. at times $$T^{\ssup i}_n=\inf\{ s>0 \colon |W^{\ssup i}_s|=e^n\}.$$ Denote by
$${\mathfrak B}_n = \bigcup_{i=1}^{k} \big\{ W^{\ssup i}_s \colon 0 \le s \le T^{\ssup i}_n\big\}$$
the union of the paths, and by $V_n$ the event that ${\mathfrak B}_n$ does not disconnect
the unit disc $\ball(0,1)$ from infinity. Then there exist constant $c_1,c_2>0$ independent of~$n$
and the starting positions, such that 
$$c_1\, n^{k} \, e^{-n \xi(2k)} \le 
\prob\big( V_n \, \big| \, T^{\ssup i}_0<T^{\ssup i}_n \mbox{ for all $1\le i \le k$ }\big) 
\le c_2\, n^{k} \, e^{-n \xi(2k)}.$$
\end{theorem}
\medskip
\pagebreak[2]

Considering that $\prob(T^{\ssup i}_0<T^{\ssup i}_n)$ is a constant multiple of $\frac 1n$, and 
applying Brownian scaling, we infer from this that~$(i)$ holds with $\xi=\xi(4)$. The fact that 
Theorem~\ref{exp} also yields the second order estimate~$(ii)$ is best explained by Figure~\ref{f1} 
below. If both squares marked~$S$ and $T$ are in $\skris_n$, then for each of the three shaded 
discs the smaller unshaded disc inside is visited by two independent pieces of the Brownian motion, 
and the union of the pieces outside the smaller discs does not disconnect the smaller discs from infinity. 
Using a scaled version of Theorem~\ref{exp} ---  once with interior radius of order~$2^{-m}$ and exterior radius
of order~$\delta$, and  twice with interior radius of order~$2^{-n}$  and exterior radius
of order~$2^{-m}$ --- and considering also the hitting probabilities, gives a bound of 
$2^{-m \xi(4)}\, 2^{-2(n-m) \xi(4)}$, which readily implies~$(ii)$.\smallskip

\begin{figure}[htbp]
  \centerline{ \hbox{ \psfig{file=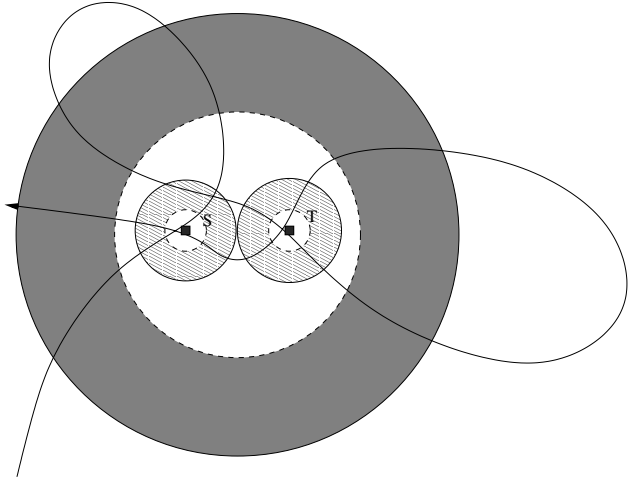,height=2.5in}}}
  \caption{The second moment estimate. The large shaded disc has radius of order~$\delta$,
  the unshaded disc inside has radius of order~$2^{-m}$, as do the two disjoint lightly shaded ones 
  inside it. The tiny discs around the squares~$S$ and~$T$ have radius of order $2^{-n}$.}
  \label{f1}
\end{figure}

These arguments show that, with positive probability,
$$\dim \big( D \cap \partial U \big) = 2- \xi(4) = \frac{\sqrt{97}+1}{24}\, .$$
To verify that this holds almost surely, we observe that $W_\tau$ is always a point on the 
frontier. The previous arguments can be adapted to show that there is a positive probability that the
double points on the frontier intersected with any small disc around this point
have the given Hausdorff dimension. Then a variant of Blumenthal's zero-one law can be applied and
yields the result with probability one.% 
\smallskip
\pagebreak[2]

Let us mention that our technique of proof can also be used to show that
the dimension of the frontier itself is~$2-\xi(2)=4/3$. However, in this case the original 
proof given by Lawler~\cite{La96} is easier. Also, the non-existence of triple points 
on the frontier follows rather easily from the fact that $\xi(6)>2$, and indeed most 
of~\cite{BW96} is devoted to the derivation of this estimate,  at a time when exact values 
of intersection exponents were not yet available.
\smallskip
\pagebreak[2]

We would further like to note that our method can be used to estimate the dimension of another 
type of sets. Burdzy and Werner conjectured~\cite{BW96} that there are no times~$t$ 
such that $W_t$ is a triple point on the boundary of the set $\{W_s \colon s\in [0,t]\}$. 
These points may be called \emph{pioneer-triple points}. Analogously defining pioneer-double and 
ordinary pioneer points of Brownian motion, our technique can be used to give the dimension of
these sets as $2-\xi(5)$, $2-\xi(3)$ and $2-\xi(1)$, the latter being already known from~\cite{LSW01}. 
Unfortunately $2-\xi(5)$ is equal to zero, so that our result neither proves nor disproves the conjecture 
of Burdzy and Werner.
\bigskip
\pagebreak[2]

\section{Proof of Theorem~\ref{double}: Details.}

%Throghout the proof we denote by~$c$ a constant the value of which may change from line to line. 
For a Brownian motion $(W_t \colon t\ge 0)$ and sets $A_1, A_2,\ldots$ we define recursively
$$\begin{aligned}\tau(A_1) & := \inf \big\{t>0 \colon W_t\in A_1\big\},\\
\tau(A_1, \ldots, A_n) & := \inf \big\{t> \tau(A_1, \ldots, A_{n-1}) \colon W_t\in A_n\big\}.
\end{aligned}$$
For any bounded set~$A$ we denote by $\ball(A,r)$ the disc of radius $r$ around the barycentre of~$A$. 
\smallskip

We keep the notation introduced in the previous section; recall in particular the meaning of the fixed 
parameters~$\delta>0$ and $R>2$, on which all constants of this section may depend.
Divide $S_0$ into its dyadic subsquares, 
$$S^{i,j}_n := [x+i2^{-n},x+(i+1)1^{-n}] \times [y+j2^{-n},y+(j+1)2^{-n}], 
\quad \mbox{ for } i,j\in\{0,\ldots, 2^n-1\},$$  
where $(x,y)\in\R^2$ denotes the bottom left corner of $S_0$. 
\smallskip

\begin{definition}
Let $N(\delta)$ be the smallest integer satisfying~$2^{-n}<\delta/16$.
For $n\ge N(\delta)$ define the collection $\mathcal S_n$ of $\delta$-good squares to be the set of all 
$S=S^{\scriptscriptstyle{i,j}}_n$ with the following properties:
\begin{itemize}
\item[(1)] $S$ is visited twice and between the visits the motion travels a distance close to $\delta$; more precisely
$$\tau(S,\partial\ball(S, \delta- \sfrac1{\sqrt{2}}\,2^{-n}), S)< \tau;$$
\item[(2)] $\ball(S,2^{-n})$ is not disconnected from infinity by the path~$\{W_s \colon s\in [0,\tau]\}$.\\[-3mm]
\end{itemize}
We write $S_n$ for the union of all $S\in \mathcal S_n$.
\end{definition}

%\begin{remark}
The difference of $(1/\sqrt{2})\, 2^{-n}$ between the subtractive corrections for $n-1$ and $n$ is exactly the distance between the centres of two dyadic
squares $S^{\scriptscriptstyle{i,j}}_n \subset S^{\scriptscriptstyle{k,l}}_{n-1}$. It ensures that if $S^{\scriptscriptstyle{i,j}}_n$ is $\delta$-good, then so is $S_{n-1}^{\scriptscriptstyle{k,l}}$. 
It is easy to see that a point~$x$ which is contained in every member of a decreasing 
sequence of $\delta$-good squares is a double point, with two visits to~$x$ separated by an excursion
reaching~$\partial\ball(x,\delta)$. We therefore have
$$S_0 \cap D \cap \partial U = \bigcup_{\delta>0} \, \bigcap_{n=N(\delta)}^\infty \, \bigcup_{S \in {\mathcal S}_n} 
\, S.$$
As explained in the previous section, the main step in the proof is to establish the following lemma,
which implies that $\dim (D \cap \partial U)\le 2-\xi(4)$ almost surely, and
$\dim (D \cap \partial U)\ge 2-\xi(4)$ with positive probability.
%\end{remark}
\pagebreak[2]

\begin{lemma}\label{main}
There exist constants $c_1,c_2,c_3>0$ such that for any $n\ge m \ge N(\delta)$ and any
dyadic subsquare $S\subset S_0$ of sidelength $2^{-n}$, we have 
$$c_1 \, 2^{-\xi(4) n} \leq \prob\big(S \in {\mathcal S}_n\big)  \leq c_2 \, 2^{-\xi(4) n},$$
and for any pair of dyadic subsquares $S,T\subset S_0$ of sidelength $2^{-n}$
with distance in $[\frac 12 2^{-m},2^{-m}]$,  we have
$$\prob\big(S,T \in {\mathcal S}_n\big) \leq c_3 \, 2^{-2\xi(4) n} \, 2^{\xi(4) m}.$$
\end{lemma}

The following three sections are devoted to the proof of this lemma. In the course of the proof,
we also provide the arguments needed to prove Theorem~\ref{exp}. The proof of Theorem~\ref{double}
is then completed in Section~3.5 by means of a zero-one argument.
\medskip

\subsection{From Brownian paths to excursions.}
Given an annulus $A={\rm cl}\,\ball(x,r_1 \vee r_2)\setminus\ball(x,r_1 \wedge r_2)$ we define
an \emph{excursion} from $\partial\ball(x,r_1)$ to $\partial\ball(x,r_2)$ as a continuous curve 
$\gamma\colon[0,\tau]\to A$ with 
$$\gamma[0,\tau] \cap \partial\ball(x,r_1)=\{\gamma(0)\}, \mbox{  and } 
\gamma[0,\tau] \cap \partial\ball(x,r_2)=\{\gamma(\tau)\}.$$ 
To define a Brownian excursion, start a Brownian motion $\{W_t \colon t\ge 0)$ uniformly 
on~$\partial\ball(x,r_1)$ and define
$\sigma= \sup\{t\le \tau(\partial\ball(x,r_2)) \colon W_t\in\partial\ball(x,r_1)\}$.
Then the random curve $(Y_t \colon 0\le t\le \tau)$ with $\tau=\tau(\partial\ball(x,r_2))-\sigma$
and $Y_t=W_{t+\sigma}$ defines a \emph{Brownian excursion} from $\partial\ball(x,r_1)$ to 
$\partial\ball(x,r_2)$. As described, for example, in~\cite{LSW02} or \cite{MS08}, 
the time-reversal of a Brownian excursion from $\partial\ball(x,r_1)$ to  
$\partial\ball(x,r_2)$ is a Brownian excursion from $\partial\ball(x,r_2)$ to $\partial\ball(x,r_1)$.
\medskip

Fix a square $S\subset S_0$ of sidelength $r_1:=2^{-n}$ and radii $r_1<r_2< r_3$ sufficiently small to ensure
$r_3<\delta$ and $r_2<\frac23\, r_3$. With such a configuration we associate natural curves and excursions embedded in a Brownian 
motion~$(W_s \colon s\ge 0)$ as follows: Let
$$\begin{aligned}
t_1^{\ssup1} & = \tau\big(\ball(S,r_2)\big), \qquad
&  t_2^{\ssup1} & = \tau\big(\ball(S,r_1)\big),\\[1mm]
t_3^{\ssup1} & = \tau\big(S, \partial\ball(S,r_1) \big), \qquad
&  t_4^{\ssup1} & = \tau\big(S, \partial\ball(S,r_2)\big),\\[1mm]
\end{aligned}$$
and define the curves
$$\begin{aligned}
W_1^{\ssup 1}\colon & [0,t_2^{\ssup 1}-t^{\ssup1}_1] \to  \R^2\setminus\ball\big(S,r_1\big),
\qquad & W^{\ssup 1}_1(t)=W_{t^{\ssup1}_1+t},\\
W_2^{\ssup 1}\colon & [0,t_4^{\ssup 1}-t^{\ssup1}_3] \to  {\rm cl}\,\ball\big(S,r_2\big),
\qquad & W^{\ssup 1}_2(t)=W_{t^{\ssup1}_3+t}.\\[1mm]
\end{aligned}$$
%and excursions
%$$\begin{aligned}
%W_1^{\ssup 1}\colon [0,t_1^{\ssup 1}-s^{\ssup1}_1] \to  {\rm cl}\,\ball\big(S,r_2\big)\setminus\ball\big(S,r_1\big),
%\qquad & W^{\ssup 1}_1(t)=W_{s^{\ssup1}_1+t},\\
%W_2^{\ssup 1}\colon [0,t_2^{\ssup 1}-s^{\ssup1}_2] \to  {\rm cl}\,\ball\big(S,r_2\big)\setminus\ball\big(S,r_1\big),
%\qquad & W^{\ssup 1}_2(t)=W_{s^{\ssup1}_2+t},\\[1mm]
%\end{aligned}$$
% are excursions from  $\partial\ball(S, r_2)$ to $\partial\ball(S,r_1)$, respectively
% from $\partial\ball(S,r_1)$ to $\partial\ball(S, r_2)$.
%\smallskip
Similarly,  we define curves associated with further visits to~$\ball(S,r_1)$. Indeed, for $i\ge 2$,
let $$t_5^{\ssup{i-1}}=\inf\big\{t\ge  t_4^{\ssup{i-1}} \colon W_t\in
\partial\ball\big(S,r_3\big) \big\},$$
and let $W^{\ssup i}_1, W^{\ssup i}_2$ be defined as before, but for the Brownian motion started at 
time~$t_5^{\ssup{i-1}}$. The next lemma states that these curves are almost independent.
\smallskip

\begin{lemma}\label{start}
Let $(X^{\ssup i}_t \colon 0\le t \le \tau^{\ssup i})$, $1\le i \le 2k$, be independent Brownian 
motions started uniformly on  $\partial\ball(S, r_2)$ and stopped upon reaching $\partial\ball(S,r_1)$, if $1\le i\le k$,
and started uniformly on $\partial\ball(S,r_1)$ and stopped upon reaching $\partial\ball(S, r_2)$, if $k<i\le 2k$.
Then the law of this family, and the joint law of the curves
$$\begin{aligned}
\big(W^{\ssup i}_1(t) & \colon 0\le t \le t_2^{\ssup i}-t^{\ssup i}_1\big), \quad i \le k; \\
\big(W^{\ssup i}_2(t) & \colon 0\le t \le t_4^{\ssup i}-t^{\ssup i}_3\big), \quad i \le k;\\[1mm]
\end{aligned}$$
are mutually absolutely continuous with densities bounded by constants, which do not depend on the choice of 
the radii~$r_1, r_2, r_3$, but may depend on the choice of~$\delta$.
\end{lemma}

\begin{proof}
By the Harnack principle, the laws of $W(t^{\ssup 1}_1)$ and $X^{\ssup 1}_0$ are absolutely
continuous with a bounded density. Moreover, conditional on these points, the curves
$W^{\ssup 1}_1$ and $X^{\ssup 1}$ have the same law. Given their endpoints $W(t^{\ssup 1}_2)$ 
and $X^{\ssup 1}_{\tau^{(1)}}$, using the Harnack principle again, the laws of $W(t^{\ssup 1}_3)$ 
and $X^{\ssup{k+1}}_0$,\vspace{-1mm} are absolutely continuous with a bounded density
and, conditional on these points, the curves $W^{\ssup 1}_2$ and $X^{\ssup {k+1}}$ have the same law.
Together with the strong Markov property, this implies that the unconditional
laws of the pairs $(W^{\ssup 1}_1, W^{\ssup 1}_2)$ and $(X^{\ssup 1}, X^{\ssup{k+1}})$ are mutually 
absolutely continuous with bounded densities. Iterating this argument further completes the proof.
\end{proof}

For the lower bounds we need to study configurations of curves, which not only
fail to disconnect, but do not even come close to doing so. To make this precise 
we introduce the notion of an $\alpha$-nice configuration, which is a relaxation 
of the same notion in~\cite{LSW02}.  
\smallskip

\begin{definition}
Suppose that $(\gamma_s^{\ssup 1} \colon 0\le s\le \tau^{\ssup 1}),\ldots, 
(\gamma_s^{\ssup k} \colon 0\le s\le \tau^{\ssup k})$  are planar curves started
on the boundary of a fixed annulus~$A$ and stopped upon reaching the opposite boundary 
circle. This configuration of curves is called \emph{$\alpha$-nice} if 
\begin{itemize}
\item[(i)] $\displaystyle\big\{\gamma^{\ssup i}_s \colon 0\le s \le \tau^{(i)}\} 
\setminus A \subset \ball(\gamma_0^{\ssup i}, \alpha \,|\gamma_0^{\ssup i}|)$, and \\[1mm]
\item[(ii)] the set
$$\bigcup_{i=1}^k \big\{\gamma^{\ssup i}_t \colon 0\le t \le \tau^{\ssup i}\big\}
\cup \bigcup_{i=1}^k \ball\big(\gamma_0^{\ssup i}, \alpha \,|\gamma_0^{\ssup i}| \big) 
\cup \bigcup_{i=1}^k \ball\big(\gamma_{\tau^{(i)}}^{\ssup i}, \alpha  \,|\gamma_{\tau^{(i)}}^{\ssup i}|\big) $$
does not disconnect the centre of the annulus from infinity.
\end{itemize}
Note that condition~(i) is void if the curves are excursions between the bounding
circles of the~annulus.
\end{definition}

As we often argue on an exponential scale, it is convenient 
to introduce the abbreviation $\circle_a$ for  $\partial\ball(0,e^a)$ and to denote by 
$\annulus(a,b)$ the annulus between the circles $\circle_a$ and $\circle_b$.
In several instances we will use that, for a planar Brownian motion started in~$x$ and $0<e^a<|x|<e^b$,
\begin{equation}\label{hitting_rho}
\prob\big( \tau(\circle_a)<\tau(\circle_b) \big) = \frac{b- \log |x|}{b-a},
\end{equation}
see~\cite[Theorem~3.17]{MP08}.
The following key lemma identifies the disconnection probabilities for Brownian excursions. Its proof is 
postponed to Section~\ref{lsw_pf}.
\smallskip

\begin{lemma}\label{lsw}
Fix a positive integer~$k$ and, for $n_1<n_2$, suppose that 
$$(Y^{\ssup i}_t \colon 0\le t \le \tau^{\ssup i}), \quad\mbox{ for } i\in\{1,\ldots, k\},$$ 
are independent Brownian excursions from~$\circle_{n_1}$ to $\circle_{n_2}$. Let 
$p(n_1,n_2,k)$ be the probability that the union of these excursions does not disconnect $\circle_{n_1}$ 
from infinity and $p_\alpha(n_1,n_2,k)$ be the probability that they form an 
$\alpha$-nice configuration. Then there exist constants $C_1, C_2>0$, independent of $n_1, n_2$, 
and an $\alpha_0>0$ such that, for every $\alpha \in [0,\alpha_0]$,
$$C_1 \, (n_2-n_1)^k\, e^{(n_1-n_2)\, \xi(k)}\leq 
{p}_\alpha(n_1,n_2,k)  \leq p(n_1,n_2,k) \leq C_2 \, (n_2-n_1)^k \, e^{(n_1-n_2)\, \xi(k)}.$$
\end{lemma}
\smallskip

The following result is the main tool from this section. It is derived from Lemma~\ref{lsw} by
extracting suitable excursions from the curves.
\medskip

\begin{lemma}\label{lowandup}
Fix integers~$0\le\ell\le k$ and, for $n_1<n_2$, suppose that 
$$(X^{\ssup i}_t \colon 0\le t \le \tau^{\ssup i}), \quad\mbox{ for } i\in\{1,\ldots, k\},$$ 
are independent Brownian motions, which in the case~$1\le i\le \ell$ are started 
uniformly in~$\circle_{n_1}$ and stopped upon reaching~$\circle_{n_2}$, and in the 
case~$\ell< i\le k$ are started uniformly in~$\circle_{n_2}$  
and stopped upon reaching~$\circle_{n_1}$. Let $q(n_1,n_2,k)$ be the probability that the union of the
$k$~paths does not disconnect $\circle_{n_1}$ from infinity, and $q_\alpha(n_1,n_2,k)$ be the probability 
that the paths form an $\alpha$-nice configuration. Then there exist constants $C_3, C_4>0$, independent of $n_1, n_2$, 
and an $\alpha_0>0$ such that, for every $\alpha \in [0,\alpha_0/2]$,
$$C_3 \, \alpha^k \, e^{(n_1-n_2)\, \xi(k)}
\leq {q}_\alpha(n_1,n_2,k)  \leq q(n_1,n_2,k) \leq C_4 \, e^{(n_1-n_2)\, \xi(k)}.$$
\end{lemma}

\begin{proof}
We start with the upper bound. Let $\sigma^{\ssup i}_0=0$ and, for $j\ge 1$, if $1\le i\le \ell$ define stopping times 
$$\begin{aligned}
\tau^{\ssup i}_j & =\inf\big\{t>\sigma^{\ssup i}_{j-1} \colon X^{\ssup i}_t\in\circle_{n_1-1}\cup\circle_{n_2}\big\},\quad
\sigma^{\ssup i}_j & =\inf\big\{t>\tau^{\ssup i}_{j} \colon X^{\ssup i}_t\in\circle_{n_1}\big\},
\end{aligned}$$
and similarly, if $\ell<i\le k$,
$$\begin{aligned}
\tau^{\ssup i}_j & =\inf\big\{t>\sigma^{\ssup i}_{j-1} \colon X^{\ssup i}_t\in\circle_{n_1}\cup\circle_{n_2+1}\big\}, \quad
\sigma^{\ssup i}_j & =\inf\big\{t>\tau^{\ssup i}_{j} \colon X^{\ssup i}_t\in\circle_{n_2}\big\}.
\end{aligned}$$
By~\eqref{hitting_rho} the random variables $N^{\ssup i}$, defined by $\tau^{\ssup i}_{N^{\ssup i}}=\tau^{\ssup i}$, are
geometric with success probability~$1/(n_2-n_1+1)$. Define the paths
$$X^{\ssup i}_j \colon [0, \sigma^{\ssup i}_j-\tau^{\ssup i}_j] \to \R^2, \qquad X^{\ssup i}_j(t)=X^{\ssup i}_{\tau^{(i)}_j+t}.$$
In particular, the paths $X^{\ssup i}_{N^{\ssup i}}$ contain an excursion from $\circle_{n_1}$ to $\circle_{n_2}$, if $1\le i\le \ell$,
or from $\circle_{n_2}$ to $\circle_{n_1}$, if $\ell<i\le k$. Using this, together with the strong Markov property, the Harnack principle,
and the time-reversibility of excursions, we obtain, for a suitable constant~$C>0$, 
$$\begin{aligned}
q(n_1,n_2,k) & \leq C\, p(n_1,n_2,k) \sum_{\ell_1,\ldots,\ell_k=1}^\infty 
\prod_{i=1}^k \prob\big( N^{\ssup i}=\ell_i\big) \\ & \qquad \times\prod_{j=1}^{\ell_i-1} \prob\big( X^{\ssup i}_j 
\mbox{ does not disconnect $\circle_{n_1}$ from infinity } \big| \tau^{\ssup i}_j<\tau^{\ssup i} \big).
\end{aligned}$$
As the factors in the second line are bounded from above by a constant $\rho<1$, we obtain
$$\begin{aligned}
q(n_1,n_2,k) & \leq C\, p(n_1,n_2,k) \, \frac1{(n_2-n_1+1)^k} \, \sum_{\ell_1=1}^\infty \rho^{\ell_1-1} \ldots \sum_{\ell_k=1}^\infty \rho^{\ell_k-1}\\
& \le C \, C_2 \,(1-\rho)^{-k} \, e^{(n_1-n_2)\, \xi(k)}.
\end{aligned}$$ 
For the lower bound we consider last exit times~$\sigma^{\ssup i}$ defined by
$$\sigma^{\ssup i} =\sup\big\{t<\tau^{\ssup i} \colon |X^{\ssup i}_t|= |X^{\ssup i}_0| \big\},$$
Observe that the intersection of the events
\begin{itemize}
\item[(1)] $\big\{X^{\ssup i}_t \colon 0\le t\le \sigma^{\ssup i}\big\} \subset 
\ball\big(X^{\ssup i}_0,\alpha\, |X^{\ssup i}_0|\big)$ 
           for all $1\le i \le k$,\\[-2mm]
\item[(2)] the set $$\bigcup_{i=1}^k \big\{X^{\ssup i}_t \colon \sigma^{\ssup i}\le t \le \tau^{\ssup i}\big\}
\cup \bigcup_{i=1}^k \ball\big(X_{\sigma^{\ssup i}}^{\ssup i}, 2\alpha \,|X_0^{\ssup i}| \big) 
\cup \bigcup_{i=1}^k \ball\big(X_{\tau^{(i)}}^{\ssup i}, 2\alpha  \,|X_{\tau^{(i)}}^{\ssup i}|\big) $$
does not disconnect $\circle_{n_1}$ from infinity,
\end{itemize}
imply that the configuration is $\alpha$-nice. By~\eqref{hitting_rho} the probability of~(1) is 
bounded from below by a constant multiple of $(\alpha/(n_2-n_1))^k$. Conditional on~(1) the paths
$\{X^{\ssup i}_t \colon \sigma^{\ssup i}\le t \le \tau^{\ssup i}\}$ are independent Brownian excursions and hence the probability of~(2) is bounded from below by a constant multiple of $p_{2\alpha}(n_1, n_2, k)$. 
Combining these two estimates and using Lemma~\ref{lsw} implies the result.
\end{proof}
\smallskip

\begin{remark} Lemma~\ref{lowandup} implies that in the definition of Brownian disconnection exponents 
we can allow that any of the paths, instead of starting in $\partial\ball(0,1)$ and being stopped on 
leaving $\ball(0,e^n)$, start on~$\partial\ball(0,e^n)$ and are stopped on hitting~$\ball(0,1)$. 
Observe also that Theorem~\ref{exp} is an immediate consequence of Lemma~\ref{lowandup}.
\end{remark}

\subsection{Proof of Lemma~\ref{main}.}
We now complete the proof of Lemma~\ref{main} using the framework provided in
the previous section. We start with the easiest part.

\begin{lemma}
There exists a constant $c_2>0$ such that, for any $n\ge N(\delta)$, and any dyadic subsquare 
$S\subset S_0$ of sidelength $2^{-n}$, we have
$$\prob\big(S \in {\mathcal S}_n\big)  \leq c_2 \,  2^{-\xi(4)n}.$$
\end{lemma}

\begin{proof}
The event $\{S\in {\mathcal S}_n\}$ implies that, for $r_1=2^{-n}$, $r_2=\frac\delta2$ and $r_3=\delta-2^{-n-\frac12}$,
the embedded paths $W^{\ssup 1}_1, W^{\ssup 1}_2, W^{\ssup 2}_1, W^{\ssup 2}_2$ do not disconnect the disc~$\ball(S,2^{-n})$
from infinity. Combining Lemma~\ref{start} and Lemma~\ref{lowandup}, for $\ell=2$ and $k=4$, gives the result.
\end{proof}

The idea of the corresponding lower bound is to describe a behaviour of Brownian motion, 
which implies the event $\{S\in {\mathcal S}_n\}$, such that all the significant probabilistic cost 
arises from making $W^{\ssup 1}_1, W^{\ssup 1}_2, W^{\ssup 2}_1, 
W^{\ssup 2}_2$ a configuration of $\alpha$-nice curves.
\medskip

\begin{lemma}\label{m=1lb}
There exists a constant $c_1>0$ such that, for any $n\ge N(\delta)$, and any dyadic subsquare 
$S\subset S_0$ of sidelength $2^{-n}$, we have
$$c_1 \,  2^{-\xi(4)n} \leq \prob\big(S \in {\mathcal S}_n\big). $$
\end{lemma}

\begin{figure}[ht]
  \centerline{ \hbox{ \psfig{file=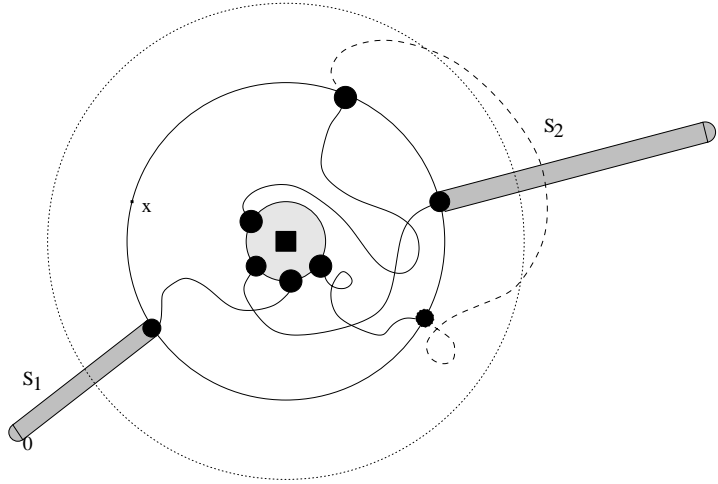,height=2.5in}}}
  \caption{Illustration of the strategy for the Brownian path explained in Lemma~\ref{m=1lb}. The three circles 
  around the solid square~$S$ have radii $r_1<r_2<r_3$. The indicated configuration of curves is $\alpha$-nice as the      
  shaded disc is not disconnected from infinity by the union of the paths and the small solid discs. The initial and       
  final parts of the path have to remain in the shaded strips, and the dashed path does not disconnect the point $x$ from infinity. }
  \label{lower}
\end{figure}

\begin{proof}
We keep the choice of $r_1=2^{-n}$, $r_2=\frac\delta2$ and $r_3=\delta-2^{-n-\frac12}$ as in the proof of the upper bound,
and fix~$0<\alpha<\alpha_0$. Define two strips ${\mathsf S}_1$ 
and ${\mathsf S}_2$ as the set of all points of distance at most~$\alpha\frac\delta2$ to the straight line 
\begin{itemize}
\item connecting the origin with the nearest point in $\partial\ball(S,\frac\delta2)$, respectively,
\item connecting $W_{t^{\ssup 2}_4}$ with the nearest point in $\partial\ball(S,2R)$.
\end{itemize}
Now look at the five events
\begin{itemize}
\item[(1)] the path $\{W_s \colon 0\le s \le t^{\ssup 1}_1\}$ remains in the strip~${\mathsf S}_1$,\\[-3mm]
\item[(2)] the path $\{W_s \colon t^{\ssup 1}_2\le s \le t^{\ssup 1}_3\}$ remains in the set
      $\ball(S,2^{-n}) \cup \ball(W_{t^{\ssup 1}_2}, \alpha 2^{-n})$,
\item[(3)] the path $\{W_s \colon t^{\ssup 2}_2\le s \le t^{\ssup 2}_3\}$ remains in the set
      $\ball(S,2^{-n}) \cup \ball(W_{t^{\ssup 2}_2}, \alpha 2^{-n})$,
\item[(4)] the path $\{W_s \colon t^{\ssup 2}_4\le s \le \tau\}$ remains in the strip~${\mathsf S}_2$,\\[-3mm]
\item[(5)] the four curves $W^{\ssup 1}_1, W^{\ssup 1}_2, W^{\ssup 2}_1, 
W^{\ssup 2}_2$ are an $\alpha$-nice configuration.
\end{itemize}
By the strong Markov property, Lemma~\ref{start} and Lemma~\ref{lowandup} the probability of the intersection of 
these five events is bounded from below by a constant multiple of $2^{-n\xi(4)}$. Given curve segments
$(W_s \colon s\in [0, t^{\ssup 1}_4] \cup [t^{\ssup 2}_1, \tau])$ satisfying these events, we may
identify a point $x\in\partial\ball(S,\frac\delta2)$ which is not disconnected from~$\ball(S,2^{-n})$ 
by the set 
$$\bigcup_{i,j=1}^2 \big\{W_t \colon t^{\ssup i}_{2j-1} \le t \le t^{\ssup i}_{2j} \big\}
\cup \bigcup_{i=1}^2 \bigcup_{j=1}^4  \ball\big(W_{t^{\ssup i}_j}, \alpha \,|W_{t^{\ssup i}_j}| \big). $$
We then additionally require
\begin{itemize}
\item[(6)] the path $\{W_s \colon t^{\ssup 1}_4\le s \le t^{\ssup 2}_1\}$ stays in
$\ball(t^{\ssup 1}_4, \alpha\frac\delta2)\cup (\R^2\setminus \ball(S,\frac\delta2))$
and also does not disconnect the point~$x$ from infinity.
\end{itemize}
Observe with the help of Figure~\ref{lower}, that under the intersection of these six events we have
$\{S\in{\mathcal S}_n\}$.  The conditional probability of the sixth event is bounded from zero, with
a bound depending on the choice of~$\delta$ and~$\alpha$. This proves the lower bound of the lemma.
\end{proof}
\medskip

In the last part we derive the second moment estimate in Lemma~\ref{main} by looking at the
path at two scales, roughly speaking, the size and the distance of the two squares~$S$, $T$.
\smallskip

\begin{lemma}
There exists a constant $c_3>0$ such that for any $n\ge m \ge N(\delta)$ and any
pair of dyadic subsquares $S,T\subset S_0$ of sidelength $2^{-n}$
with distance in $[\frac 12 2^{-m},2^{-m}]$,  we have
$$\prob\big(S,T \in {\mathcal S}_n\big) \leq c_3 \, \cdot 2^{-m\xi(4)} \cdot 2^{-2(n-m)\xi(4)}.$$
\end{lemma}

\begin{proof}
We denote by $z$ be the middle point between the centres of $S$ and $T$.
If $S,T\in{\mathcal S}_n$ we know that the Brownian path visits the sets $S, \partial\ball(z, \sfrac\delta2), S$
and $T,\partial\ball(z, \sfrac\delta2), T$, both in that order, before exiting $\ball(0,R)$, but there are 
eight possible combinations of these events, not counting possible additional visits of $\partial\ball(z, \sfrac\delta2)$. 
These can be described symbolically as follows:
\begin{align*}
E_1: \qquad & S \leadsto\partial\ball(z, \sfrac\delta2)\leadsto S \leadsto T \leadsto\partial\ball(z, \sfrac\delta2)\leadsto T\\
E_2: \qquad & S \leadsto\partial\ball(z, \sfrac\delta2)\leadsto T\leadsto S\leadsto \partial\ball(z, \sfrac\delta2)\leadsto T\\
E_3: \qquad & S \leadsto T\leadsto \partial\ball(z, \sfrac\delta2)\leadsto S\leadsto T\\
E_4: \qquad & S \leadsto T\leadsto \partial\ball(z, \sfrac\delta2)\leadsto T\leadsto S\\
%\end{align*}
%\begin{align*}
E_5: \qquad & T \leadsto \partial\ball(z, \sfrac\delta2)\leadsto T\leadsto S\leadsto \partial\ball(z, \sfrac\delta2)\leadsto S\\
E_6: \qquad & T \leadsto \partial\ball(z, \sfrac\delta2)\leadsto S\leadsto T\leadsto \partial\ball(z, \sfrac\delta2)\leadsto S\\
E_7: \qquad & T \leadsto S\leadsto \partial\ball(z, \sfrac\delta2)\leadsto T\leadsto S\\
E_8: \qquad & T \leadsto S\leadsto \partial\ball(z, \sfrac\delta2)\leadsto S\leadsto T
\end{align*}
Note that, owing to possible additional visits, these events are not disjoint.
Each of the events allows a similar estimate, and for notational convenience we focus here on 
the event~$E_4$, which is satisfied in the case sketched in Figure~\ref{f1}. 
We first assume that $n\ge m+4$. In this case we define an increasing 
sequence of twenty-four stopping times:
$$\begin{aligned}
t_1 & = \tau\big(\ball(z,\sfrac\delta2)\big), \qquad
&  t_2 & = \tau\big(\ball(z,2^{-m+2})\big),\\[1mm]
t_3 & = \tau\big(\ball(S,2^{-m-3})\big), \qquad
&  t_4 & = \tau\big(\ball(S,2^{-n})\big),\\[1mm]
t_5 & = \tau\big(S, \partial\ball(S,2^{-n})\big), \qquad
&  t_6 & = \tau\big(S, \partial\ball(S,2^{-m-3})\big),\\[1mm]
t_7 & = \tau\big(S, \ball(T,2^{-m-3})\big), \qquad
&  t_8 & = \tau\big(S, \ball(T,2^{-n})\big),\\[1mm]
t_9 & = \tau\big(S,T, \partial\ball(T,2^{-n})\big), \qquad
&  t_{10} & = \tau\big(S,T, \partial\ball(T,2^{-m-3})\big),\\[1mm]
t_{11}  & = \tau\big(S, T, \partial\ball(z, 2^{-m+2}) \big), \qquad
&  t_{12} & = \tau\big(S, T, \partial\ball(z,\sfrac\delta3)\big),\\[1mm]
t_{13} &  =\tau\big(S, T, \partial\ball(z,\sfrac\delta2)\big), &
t_{14}  & = \tau\big(S,T, \partial\ball(z,\sfrac\delta2), \ball(z, 2^{-m+2})\big), \\[1mm]
t_{15} & = \tau\big(S,T, \partial\ball(z,\sfrac\delta2), \ball(T,2^{-m-3})\big), \qquad
&  t_{16} & = \tau\big(S,T, \partial\ball(z,\sfrac\delta2), \ball(T,2^{-n})\big),\\[1mm]
t_{17} & = \tau\big(S,T, \partial\ball(z,\sfrac\delta2), T, \partial\ball(T,2^{-n})\big), \qquad
&  t_{18} & = \tau\big(S,T, \partial\ball(z,\sfrac\delta2), T, \partial\ball(T,2^{-m-3})\big),\\[1mm]
t_{19} & = \tau\big(S,T, \partial\ball(z,\sfrac\delta2), T,\ball(S,2^{-m-3})\big), \qquad
&  t_{20} & = \tau\big(S,T, \partial\ball(z,\sfrac\delta2), T,\ball(S,2^{-n})\big),\\[1mm]
t_{21} & = \tau\big(S,T, \partial\ball(z,\sfrac\delta2), T,S, \partial\ball(S,2^{-n})\big), \qquad
&  t_{22} & = \tau\big(S,T, \partial\ball(z,\sfrac\delta2), T,S, \partial\ball(S,2^{-m-3})\big),\\[1mm]
t_{23}  & = \tau\big(S,T, \partial\ball(z,\sfrac\delta2), T,S, \partial\ball(z, 2^{-m+2}) \big), \qquad
&  t_{24} & = \tau\big(S,T, \partial\ball(z,\sfrac\delta2),T,S, \partial\ball(z,\sfrac\delta2)\big).\\[1mm]
\end{aligned}$$
For $1\le j\le 12$ we define the curves
$$W_{2j-1}\colon  [0,t_{2j}-t_{2j-1}] \to  \R^2, \qquad  W_{2j-1}(t)  =W_{t_{2j-1}+t}.$$
Although the twelve curves we have now defined are not independent, arguing with 
the Harnack principle as in Lemma~\ref{start} shows that this may be assumed at the expense
of a constant multiplicative factor. The event $S,T\in{\mathcal S}_n$ implies that, on the large
scale, the  curves $W_1,W_6,W_7,W_{12}$ do not disconnect $\ball(z,2^{-m+2})$ from infinity, and, 
on the small scale, the curves $W_2, W_3, W_{10}, W_{11}$ do not disconnect $\ball(S,2^{-n})$ from
infinity, and the curves $W_4, W_5, W_8, W_9$ do not disconnect $\ball(T,2^{-n})$ from infinity. 
Using Lemma~\ref{lowandup} now shows that, for a suitable constant $C>0$,
$$\prob\big(\big\{S,T\in{\mathcal S}_n\big\} \cap E_4 \big) \le C\, 
\, \cdot 2^{-m\xi(4)} \cdot 2^{-2(n-m)\xi(4)}.$$
This also holds in the degenerate case $n<m+4$, in which we can neglect the small scale
and apply the first moment bound to a square with sidelength of order $2^{-m}$, which 
contains both $S$ and~$T$. The final result follows now by summing over all estimates 
for $E_1,\ldots, E_8$, in which we only have to define the stopping times suitably, in 
order to avoid an overlap of the intervals defining the twelve curves.
\end{proof}

\subsection{Proof of Lemma~\ref{lsw}}\label{lsw_pf}
We follow the lines of~\cite{LSW02}. Recall that $\annulus(0,n)$ 
denotes the annulus between the circles~$\circle_0$ and 
$\circle_n$ and suppose that $(Y^{\ssup i}_t \colon 0\le t\le \tau^{\ssup i})$ 
are independent Brownian excursions across $\annulus(0,n)$ for $i \in \{1,\ldots, k\}$. The set 
$$\annulus(0,n)\setminus\bigcup_{i=1}^k \big\{ Y^{\ssup i}_t \colon 0\le t\le \tau^{\ssup i} \big\}$$ 
contains at most $k$ connected components joining  $\circle_0$ and $\circle_n$. 
As defined in~\cite{LSW02} we let $L_n(j)$ be $\pi$ times the extremal distance between 
$\circle_0$ and $\circle_n$ in the $j$-th of these components and $L_n^k$ be the minimum of these 
extremal distances over all $1\leq j\leq k$. The crucial fact about extremal distances used here
is that $L_n^k$ is finite if and only if the excursions do not disconnect $\circle_0$ from infinity. 
\medskip

The following lemma is Theorem~3.1 in~\cite{LSW02}.

\begin{lemma}\label{lsw1}
For any $\lambda_0 >0$ and any $k\in \N$, there exist $\xi(k,\lambda)$ such that 
$$ \E \left[e^{-\lambda L_n^k}\right] \asymp n^k e^{-n\xi(k,\lambda)},$$
where $\asymp$ means that the ratio of the two sides is bounded away from zero
and infinity by constants not depending on the choice of
$\lambda \in (0,\lambda_0]$ and $n\geq 1$.
\end{lemma}

Letting $\lambda\downarrow 0$ in the previous lemma shows that the probability
that the excursions do not disconnect $\circle_0$ from infinity equals
$$\begin{aligned}
\prob\big( L_n^k<\infty \big)
& = \lim_{\lambda\downarrow 0} \E \left[e^{-\lambda L_n^k}\right]
\asymp n^k e^{-n\xi(k)},
\end{aligned}$$
where $\xi(k)=\lim_{\lambda\downarrow 0} \xi(k,\lambda)$. In \cite{LSW02} it is shown,
using a similar argument for Brownian motion in place of Brownian excursion, that
this $\xi(k)$ is also the disconnection exponent for Brownian motion as defined in our
framework.
\medskip

We now adapt Lemma 4.1 in~\cite{LSW02} for our purpose, remembering that our notion of 
an $\alpha$-nice configuration is relaxed compared to the notion in~\cite{LSW02}.

\begin{lemma}\label{lsw2}
There exists an $\alpha_0>0$ such that, for any $\lambda_0>0$,  
$$ \E \left[e^{-\lambda L_n^k} \one_{\alpha\mbox{\scriptsize{-nice}}}\right] 
\asymp n^k e^{-n\xi(k,\lambda)},$$
where the implied constants are not depending on the choice of
$\lambda \in (0,\lambda_0]$, $\alpha\in (0,\alpha_0]$ and $n\geq 1$.
\end{lemma}

\begin{proof}
We assume that $n \ge 3$, and recall that Lemma~4.1 in~\cite{LSW02} implies that for any $j\leq k$, $\eps>0$ there 
is $\alpha_0>0$ such that, for all $0<\alpha \le\alpha_0$,
$$\E \left[e^{-\lambda L_n(j)}\right]-\E \left[e^{-\lambda L_n(j)} \one_{\alpha\mbox{\scriptsize{-nice}}}\right] 
\leq  \eps \, \E \left[e^{-\lambda L_{n-2}^k}\right].$$ 
Hence, using Lemma~\ref{lsw1} in the last step,
\begin{eqnarray*}
\E \left[e^{-\lambda L_n^k} \one_{\mbox{\scriptsize{not$\,\alpha$-nice}}}\right]
&\leq& \sum_{j=1}^k \Big( \E \big[e^{-\lambda L_n(j)}\big] - \E \big[e^{-\lambda L_n(j)}
\one_{\mbox{\scriptsize{$\alpha$-nice}}}\big] \Big)\\
&\leq& k \, \eps \, \E \left[e^{-\lambda L_{n-2}^k}\right]\\
& \leq & n^k e^{-n\xi(k,\lambda)} \times \Big(\eps\, k \, c_2 \left(1-\sfrac 2n\right)^k 
e^{2\xi(k,\lambda)} \Big).
\end{eqnarray*}
The bracket can be made arbitrarily small by choice of~$\eps>0$. The result now follows by combining
this inequality with Lemma~\ref{lsw1}.
\end{proof}
\medskip

The proof of Lemma~\ref{lsw} is completed by the following lemma together with
Brownian scaling.
\medskip

\begin{lemma}
Suppose $(Y^{\ssup i}_t \colon 0\le t \le \tau^{\ssup i})$ are independent Brownian excursions 
across~$\annulus(0,n)$ for $i \in \{1,\ldots, k\}$. 
Then there exists an $\alpha_0>0$ such that, for every $\alpha \in [0,\alpha_0]$,
\begin{eqnarray*}
\prob\big( L_n^k < \infty \big)  &\asymp&  n^k e^{-n\xi(k)}, \ \ \ \ \ \ \ \ \  \mbox{ and }\\
\prob\big( \{L_n^k < \infty\} \cap \{ \alpha \mbox{-nice} \} \big)  &\asymp& n^k e^{-n\xi(k)}.
\end{eqnarray*}
\end{lemma}

\begin{proof} We use monotone convergence and Lemma~\ref{lsw1} to see
\begin{eqnarray*}
\prob\big( L_n^k < \infty \big) 
&=& \lim_{\lambda \downarrow 0}\E \left[ e^{-\lambda L_n^k}\right] \asymp \lim_{\lambda \downarrow 0}n^k e^{-n\xi(k,\lambda)} = n^k e^{-n\xi(k)},
\end{eqnarray*} and the second estimate is proved the same way using Lemma~\ref{lsw2}.
\end{proof}

\subsection{The zero-one law.} It remains to show that $\dim (D\cap \partial U) \ge 2- \xi(4)$ not only with 
positive probability, but actually with probability one. To do this, we need to identify
a point on the frontier and establish a variant of Blumenthal's zero-one law for the germ-$\sigma-$algebra 
of Brownian motion around that point. In our context it is best to use the endpoint~$W_\tau$ of the path, 
which is easily seen to be always on the frontier. 
\smallskip

The time reversal of $( W_t \colon t\in [0,\tau])$ for which, 
by rotational invariance, we may assume that $W_\tau =-1$, is the image under the conformal map 
$$f\colon \H \to \ball(0,1), \quad  f(z)= \frac{z-\im}{z+\im}$$
of a \emph{half-plane excursion} started at zero. Half-plane excursions, as discussed in~\cite{La05},
can be written as $(Y_t \colon t\ge 0)$ with $Y_t:= B_t + \im \hat{B}_t$ where 
$(B_t \colon t\ge 0)$ is a real Brownian motion and $(\hat{B}_t \colon t\ge 0)$ an independent 
three-dimensional Bessel process both started at zero.
By the conformal invariance of Brownian excursions (up to time change) and the fact that conformal mappings 
preserve the Hausdorff-dimension of sets, it will be sufficient to consider the lower bound on the dimension of 
the double points on the frontier for a half-plane excursion in neighbourhoods of zero. We therefore
now denote by~$D$ the set of double points of a half-plane excursion $(Y_t \colon t\ge 0)$.
% We prove the zero-one law in this setting.
\smallskip
\pagebreak[1]

Let $I(a):=\{z \in \C \colon \Im(z) \in [0,a)\}$ and  $J(a):=\{z \in \C \colon \Im(z) \in [a,\infty)\}$ and we denote 
by $T_a$ the first hitting time of $J(a)$, i.e. $T_a := \inf \{ s>0 : Y_s \in J(a)\}$. For a set $A\subset \H$, we  write $U(A)$ for the union of the unbounded connected components of $\H \setminus A$.
We initially focus on the half-plane excursion up to time~$T_{3b}$, where $b>0$ is some small constant, and require the 
following variant of our main result so far.

\begin{lemma}\label{01-1}
For every $b>0$ there is a positive probability that
$$\dim \big(\ball(0,b) \cap D \cap \partial U( \{Y_s \colon s\in [0,T_{3b}]\} \cup  J(2b)) \big) \geq 2-\xi(4).$$
\end{lemma}

Let us first argue how to complete the proof using Lemma~\ref{01-1}. Note first that the probability in Lemma~\ref{01-1} is independent of~$b$, which is clear by scaling invariance. Also, by the transience of the excursion, there is
a positive probability, independent of $b$, that once the excursion has reached $J(3b)$ it never 
visits $I(2b)$ again. This implies that, for every $b$, the event
$$A_b := \left\{ \dim \big( \ball(0,b)\cap D \cap \partial U( \{ Y_s \colon s \in[0,T_\infty] \} ) \big) 
\geq 2-\xi(4) \right\}$$
has a positive probability $p$, independent of~$b$. It is easy to see that, for $b'<b$, we have $A_{b'} \subset A_b$ and 
therefore $\bigcap_{b>0} A_b$ is an event of the germ-$\sigma$-algebra of the half-plane excursion, which is trivial by 
Blumenthal's zero-one law. %for Brownian excursions. 
Hence this intersection and all events $A_b$ must have probability one, because
$$\prob \Big(\bigcap_{b>0} A_b\Big)=\inf_{b>0} \prob\left(A_b\right) =p>0.$$
So it remains to show Lemma~\ref{01-1}, and by scaling invariance it suffices to discuss
the case~$b=2$. The following lemma is a variant of the lower bound proved in the previous sections.

\begin{lemma}\label{01-2} Let $(W_t \colon t\ge 0)$ be a planar Brownian motion started in some point~$z$ 
and stopped at the first hitting time~$T$ of the circle~$\partial\ball(z,\frac12)$. Fix $\gamma=\frac1{10}$ and
an arbitrary square $S_0$ of sidelength~$\frac 18$ in the upper half of $\ball(z,\frac12)$ with distance 
more than $2\gamma$ from both the horizontal line through~$z$ and the circle~$\partial\ball(z, \frac 12)$. Let 
$$\Gamma :=  \{W_s \colon 0\le s \le T\}  \cup \ball(W_T,\gamma) \cup 
\big\{x\in \ball(z,\sfrac12) \colon \Im(x) \leq \Im(z)\big\},$$
and define  $\partial U(\Gamma)$ to be the boundary of the unbounded component of the complement of~$\Gamma$.
Then, with positive probability, %$\Im(W_T)> \Im(z)$ and
$$\dim \left(S_0\cap D \cap \partial U(\Gamma) \right) \geq 2-\xi(4).$$
\end{lemma}

We now divide the strong Markov process~$\{Y_t \colon t\in[0,T_6]\}$ into three parts: First the part up to 
the first hitting time~$T_1$ of $J(1)$, second the part from~$T_1$ up to the time $T$ when the process has moved a 
distance of $1/2$ from its starting point~$Y_{T_1}$, and third the remaining part starting from $T$ 
up to the first hitting time $T_6$ of~$J(6)$. For the three parts we require the following events:
\begin{itemize}   
\item[(1)] The first part remains in a small vertical strip around its starting point, more precisely 
$$\{Y_t \colon t\in[0,T_1]\} \subset {\mathsf S}:=\{z\in\H \colon \Re(z) \in [-\gamma,\gamma]\}.$$
\item[(2)] The second part satisfies %$\Im(Y_T)>1$ and 
$\dim \left(S_0\cap D \cap \partial U(\Gamma) \right) \geq 2-\xi(4)$
where the implied sets are defined as in Lemma~\ref{01-2} for the process
$(Y_{T_1+t} \colon t\ge 0)$ in place of the Brownian motion.\\[-2mm]
\item[(3)] The third part intersects neither the strip $\mathsf S$ nor the disc~$\ball(Y_{T_1}, \frac 12)$ 
except possibly inside the ball~$\ball(Y_{T}, \gamma)$.
\end{itemize}

\begin{figure}[ht]
  \centerline{ \hbox{ \psfig{file=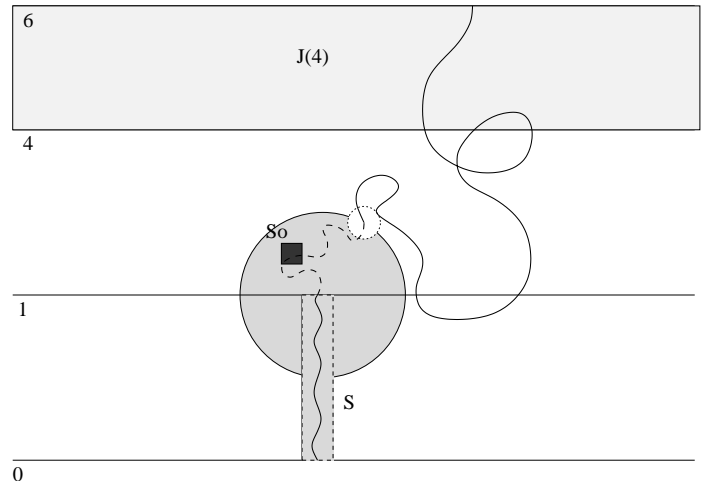,height=2.5in}}}
  \caption{Any point in $S_0\cap \partial U(\Gamma)$ is also in
  $U(\{Y_t \colon t\in[0,T_{6}]\} \cup J(4)))$, as the third part avoids
  the set $\{Y_t \colon t\in[0,T_1]\}$ completely and hits the set
  $\{Y_t \colon t\in[T_1,T]\}$, i.e. the dashed second part of the path,  
  only inside $\ball(Y_{T}, \gamma)$.}
  \label{pic}
\end{figure}

Observe, possibly with the help of Figure~\ref{pic}, that under the intersection of these three events we have
$$\dim(S_{0}\cap D \cap \partial U(\{Y_t \colon t\in[0,T_{6}]\} \cup J(4))) \geq 2-\xi(4).$$ 
Moreover, the three events, and by the strong Markov property also their intersection, have positive probability.
Indeed, for events~(1) and~(3) this is obvious. For event~(2) recall from \cite[5.3]{La05} that, given $Y_{T_1}$, 
the process $(Y_{T_1+s} \colon s\in[0,T_2-T_1])$ is 
distributed like an ordinary Brownian motion conditioned to hit $J(2)$ before the real line. It is therefore absolutely 
continuous with respect to Brownian motion and the claim follows from Lemma~\ref{01-2}.
This completes the proof of Lemma~\ref{01-1}.
\bigskip

{\bf Acknowledgements:} This paper is based on material from the first author's 
PhD thesis. We would like to thank Heinrich v.~Weizs\"acker for many helpful discussions.

\end{document}